%% file: PUPaper.tex
\documentclass[final]{siamart0516}
\usepackage{subfig}
\usepackage{forest}
\usepackage[section]{placeins}
\usepackage[export]{adjustbox}
\usepackage{changes}

\newcommand{\lp}{\left(}
\newcommand{\rp}{\right)}

\newcommand\supp{\mathop{\rm supp}}

\newcommand{\WRP}{\par\qquad\(\hookrightarrow\)\enspace}
\newcommand{\ARP}{\par\qquad\ \enspace}

\newcommand{\nmax}{n_{\text{max}}}
\newcommand{\child}[1]{c\textsubscript{#1}}
\newcommand{\weight}[1]{w\textsubscript{#1}}

\input{ex_shared}

\ifpdf
\hypersetup{
  pdftitle={\TheTitle},
  pdfauthor={\TheAuthors}
}
\fi

\begin{document}

\maketitle

\begin{abstract}
For a function that is analytic on and around an interval, Chebyshev polynomial interpolation provides spectral convergence. However, if the function has a singularity close to the interval, the rate of convergence is near one. In these cases splitting the interval and using piecewise interpolation can accelerate convergence. Chebfun includes a splitting mode that finds an optimal splitting through recursive bisection, but the result has no global smoothness unless conditions are imposed explicitly at the breakpoints. An alternative is to split the domain into overlapping intervals and use an infinitely smooth partition of unity to blend the local Chebyshev interpolants. A simple divide-and-conquer algorithm similar to Chebfun's splitting mode can be used to find an overlapping splitting adapted to features of the function. The algorithm implicitly constructs the partition of unity over the subdomains. This technique is applied to explicitly given functions as well as to the solutions of singularly perturbed boundary value problems.

\end{abstract}



\section{Introduction}
Chebyshev polynomial interpolants provide powerful approximation properties, both in theory and as implemented in practice by the Chebfun software system~\cite{battles2004extension}. Chebfun uses spectral collocation to provide very accurate automatic solutions to differential equations~\cite{driscoll2008chebop}. The method is not fully adaptive, though, since the refinement is limited to the degree of the global interpolant. 

 Chebfun includes a \textit{splitting} method that creates piecewise polynomial approximations \cite{pachon2010piecewise}. When splitting is enabled, if a Chebyshev interpolant is unable to represent the function accurately at a specified maximum degree on an interval, the interval is bisected; this process is recursively repeated on the subintervals. Afterwards adjacent subintervals are merged if the new interval allows for a Chebyshev approximation with lower degree. In effect, the method does a binary search for a good splitting location. In \cite{driscoll2014optimal} it was shown that the splitting locations are roughly optimal based on the singularity structure of the function in the complex plane. 

A drawback of Chebfun's splitting approach is that the resulting representation does not ensure anything more than $C^0$ continuity. Differentiation of the Chebyshev interpolation polynomial of degree $n$ has norm $O(n^2)$, so a jump in the derivative develops across a splitting point and becomes more pronounced for higher derivatives and larger $n$. In order to solve a boundary-value problem, Chebfun imposes explicit continuity conditions on the solution to augment the discrete problem. This solution works well in 1D but becomes cumbersome in higher dimensions, particularly if refinements are made nonconformingly.

In this paper we explore the use of Chebyshev interpolants on overlapping domains combined using a \emph{partition of unity}. The resulting approximation has the same accuracy as the individual piecewise interpolants. We use compactly supported weight functions that are infinitely differentiable, so the resulting combined interpolant is also infinitely smooth (though not analytic). We also show that the accuracy of the derivative can be bounded by $\Theta(\delta^{-2})$ for an overlap amount $\delta$, revealing an explicit tradeoff between efficiency (smaller overlap and more like Chebfun splitting) and global accuracy of the derivative. Because the global approximation is smooth, there are no matching conditions needed to solve a BVP, and there are standard preconditioners available that should aid with iterative methods for large discretizations. For example, since we split the interval into overlapping domains we could use the restricted additive Schwarz preconditioner \cite{doi:10.1137/S106482759732678X}.

We describe a recursive, adaptive algorithm for creating and applying a partition of unity, modeled on the recursive splitting in Chebfun but merging adjacent subdomains aggressively in order to keep the total node count low. Even though each node of the recursion only combines two adjacent subdomains, we show that the global approximant is also a partition of unity. We demonstrate that the adaptive refinement is able to resolve highly localized features of an explicitly given function and of a solution to a singularly perturbed BVP. 
 
The use of a partition of unity in our approximation affords us some flexibility; we are able to create approximations which are both efficient and infinitely smooth without matching. Partition of unity schemes have been widely used for interpolation \cite{franke1980smooth,mclain1976two,shepard1968two} and solving PDE's \cite{griebel2000particle,safdari2015radial}. In section~\ref{PUM_FORM_SEC} we introduce the partition of unity method, and we discuss the convergence of the method for a simple split on the interval $[-1,1]$ in section~\ref{converge_sec}. We describe our adaptive algorithm in section~\ref{PUM_recurse}. In section~\ref{PUM_BVP_SEC} we explain how to apply our method to solve boundary value problems on an interval and perform some experiments with singularly perturbed problems.
 
\section{Chebyshev interpolation}
\label{sec_cheb}
We use Chebyshev interpolants for our partition of unity method because they enjoy spectral convergence. Suppose that $f(x)$ is analytic inside a  Bernstein ellipse $E_\rho$ (an ellipse with foci $\pm 1$ and semi-major axis $\rho>1$). We then have Theorem 6 from \cite{trefethen2000spectral}:
\begin{theorem} Suppose $f(z)$ is analytic on and inside the Bernstein ellipse $E_\rho$. Let $p_n$ be the polynomial that interpolates $f(z)$ at $n+1$ Chebyshev points of the second kind. Then there exists a constant $C>0$ such that for all $n>0$,
$$ \left \| f(x)-p_n(x) \right \|_{\infty} \leq C  \rho^{-n}.$$
 \end{theorem}
If $f(x)$ is Lipschitz continuous on $[-1,1]$ then
\begin{equation}
f(x) = \sum_{k=0}^\infty a_k T_k(x), \quad a_k = \frac{2}{\pi} \int_{-1}^1 \frac{f(x) T_k(x)}{\sqrt{1-x^2}} dx,
\end{equation}
where $T_k$ denotes the degree $k$ Chebyshev polynomial (and for $a_0$, we multiply by $\frac{1}{\pi}$ instead of $\frac{2}{\pi}$). Furthermore if $p_n(x)$ is the $n$th degree Chebyshev interpolant then
\begin{equation}
f(x)-p_n(x) = \sum_{k=n+1}^{\infty} a_k \lp T_k(x)-T_m(x)\rp,
\end{equation}
where
\begin{equation}
m = \left [ (k+n-1)(\text{mod }2n) - (n-1)\right ],
\end{equation}
implying we can determine the accuracy of the interpolant $p_n(x)$ by inspecting the Chebyshev coefficients \cite{Trefethen2013}. Chebfun's standardChop method determines the minimum required degree by searching for a plateau of low magnitude coefficients \cite{Aurentz:2017:CCS:3034774.2998442}.  For example, Figure~\ref{Coeff_example} shows the first 128 coefficients of $f(x)=\exp \lp \sin \lp \pi x \rp \rp$. We see that all coefficients after the first 46 have magnitude less than $10^{-15}$. In this case, Chebfun determines the ideal degree to be 50.

\begin{figure}[!htb]
\centering
\includegraphics[scale = 0.5]{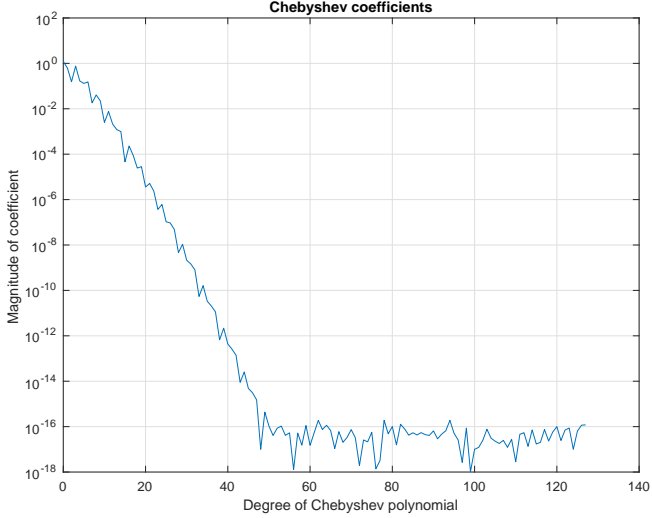}
\caption{Chebyshev coefficients for $f(x)=\exp \lp \sin \lp \pi x \rp \rp$.}
\label{Coeff_example}
\end{figure}

\section{Partition of unity formalism}
\label{PUM_FORM_SEC}
 Suppose we have an overlapping covering $\{ \Omega_k \}_{k=1}^N$ on a bounded region $\Omega$. A partition of unity is a collection of real valued functions $\{w_k(x)\}_{k=1}^N$ such that:
\begin{itemize}
\item $w_k(x)$ has support within $\Omega_k$,
\item each $w_k(x)$ is nonnegative,
\item $\forall x \in \Omega, \quad \sum_{k=1}^N w_k(x)=1$.
\end{itemize}
The functions $\{w_k(x)\}_{k=1}^N$ are called the \textit{weights} of the partition. Suppose now that $\Omega=[-1,1]$ and each $\Omega_k$ is an interval. We can use the partition of unity $\{w_k(x)\}_{k=1}^N$ to construct an approximating function. Suppose that for $m \geq 0$ we have a function $f \in C^{m}([-1,1])$, each weight $w_k(x)\in C^{m}([-1,1])$ and for each patch $\Omega_k$ we have an approximation $s_k(x)$ of $f(x)$. Then the function
\begin{equation}
\label{POUAPPROX}
s(x) = \sum_{k=1}^N w_k(x)s_k(x)
\end{equation}
can be used to approximate $f(x)$ and its derivatives \cite{wendland2004scattered}.

\begin{theorem}
\label{PUMCON}
Suppose $f \in C^{m}([-1,1])$ and for each patch $\Omega_k$ we have a function $s_k(x)$ such that
$$ \|f^{(\alpha)}(x)-s_k^{(\alpha)}(x)\|_{L_{\infty}(\Omega_k)} \leq \varepsilon_k(\alpha) $$
for $\alpha \leq m$. Thus for $j\leq m$, if $s(x)$ is the approximation (\ref{POUAPPROX}) then
\begin{equation}
\left \|f^{(j)}(x)- s^{(j)}(x) \right \|_{L_{\infty}(\Omega_k)} \leq \sum_{k=1}^N\sum_{i=0}^j \binom{j}{i} \left \| w_k^{(j-i)}(x) \right \|_{L_{\infty}(\Omega_k)} \epsilon_k(i).
\end{equation}
\end{theorem}
\begin{proof}
Since $\sum_{k=1} w_k(x)=1$, $\sum_{k=1} w_k(x)f(x)=f(x)$. Thus
\begin{equation}
\begin{aligned}
\frac{d^{j}}{d x^j}f(x)-\frac{d^{j}}{d x^j} \sum_{k=1}^N w_k(x)s_k(x) &= \frac{d^{j}}{d x^j} \sum_{k=1}^N w_k(x)(f(x)-s_k(x)) \\
&= \sum_{k=1}^N\sum_{i=0}^j \binom{j}{i} w_k^{(j-i)}(x) \lp f^{(i)}(x)-s_k^{(i)}(x) \rp.
\end{aligned}
\end{equation}
The result follows from here by the triangle inequality.
\end{proof}

\section{Convergence analysis}
\label{converge_sec}
In this section we consider a single interval partitioned into two overlapping parts, i.e.
$[-1,t]$,$[-t,1]$, where $t$ is the overlap parameter such that $0<t<1$. For the weights, we use Shepard's method \cite{shepard1968two} based on the compactly supported, infinitely differentiable shape function
\begin{align}
\psi(x) = \begin{cases}
\exp \lp  1 - \frac{1}{1-x^2}\rp & |x| < 1, \\
0 & |x| \geq 1.
\end{cases}
\end{align}
We define support functions
\begin{align}
\psi_{\ell}(x) = \psi \lp \frac{x+1}{1+t} \rp \quad \text{ and } \quad \psi_{r}(x) = \psi \lp \frac{x-1}{1+t} \rp ,
\end{align}
to construct the PU weight functions
\begin{align}
w_{\ell}(x) = \frac{\psi_{\ell}(x)}{\psi_{\ell}(x)+\psi_{r}(x)} \quad \text{and} \quad w_{r}(x) = \frac{\psi_{r}(x)}{\psi_{\ell}(x)+\psi_{r}(x)},
\label{PUW}
\end{align}
where $w_{\ell}(x),w_{r}(x)$ have support on the left and right intervals, respectively.

Suppose that $s_{\ell}(x),s_{r}(x)$ approximate $f(x)$ on $[-1,t]$, $[-t,1]$ respectively and are both infinitely smooth. Let
\begin{equation}
s(x) = w_{\ell}(x)s_{\ell}(x)+w_{r}(x)s_{r}(x),
\label{PUM2}
\end{equation}
where $s(x)$ is the partition of unity approximation. Following Theorem~\ref{PUMCON} we have for $x \in [-1,1]$ that
\begin{equation}
\begin{aligned}
\left | f(x)-s(x) \right | &= \left | w_{\ell}(x) \lp f(x) - s_{\ell}(x) \rp + w_{r}(x) \lp f(x) - s_{r}(x) \rp \right | \\
&\leq w_{\ell}(x) \left | f(x) - s_{\ell}(x) \right | + w_{r}(x) \left | f(x) - s_{r}(x) \right |.
\end{aligned}
\end{equation}
We conclude that
\begin{align}
\left \| f(x)-s(x) \right \|_{L_{\infty}[-1,1]} \leq \max \lp \left \| f(x)-s_{\ell}(x) \right \|_{L_{\infty}[-1,t]} , \left \| f(x)-s_{r}(x) \right \|_{L_{\infty}[-t,1]} \rp.
\label{POU_UP}
\end{align}
This implies that the Partition of unity Method (PUM) preserves the accuracy of its local approximants. We also have that $s(x)$ is infinitely smooth. For the first derivative we have
\begin{equation}
\begin{aligned}
\left | f'(x)-s'(x) \right | &\leq \left | w_{\ell}(x) \lp f'(x)-s_{\ell}'(x) \rp \right |+\left | w_{r}(x) \lp f'(x)-s_{r}'(x) \rp \right | \\
&+\left | w_{\ell
}'(x) \lp f(x)-s_{\ell}(x) \rp \right | +\left | w_{r}'(x) \lp f(x)-s_{r}(x) \rp \right |,
\end{aligned}
\end{equation}
giving us
\begin{equation}
\begin{aligned}
\left \| f'(x)-s'(x) \right \|_{L_{\infty}[-1,1]} &\leq \max \lp \left \| f'(x)-s_{\ell}'(x) \right \|_{L_{\infty}[-1,t]} , \left \| f'(x)-s_{r}'(x) \right \|_{L_{\infty}[-t,1]} \rp \\
&+\left \| w_{\ell}'(x) \right \|_{L_{\infty}[-t,t]} \left \| f(x)-s_{\ell}(x)\right \|_{L_{\infty}[-t,t]}\\ 
&+ \left \| w_{r}'(x) \right \|_{L_{\infty}[-t,t]} \left \| f(x)-s_{r}(x)\right \|_{L_{\infty}[-t,t]},
\end{aligned}
\label{diff_error}
\end{equation}
since the derivatives of the weights have support only on the overlap. For $t \ll 1$, the weights steepen to become nearly step functions. This causes the derivatives of the weights to be large in magnitude, resulting in an increase in the error for the derivative.

Since $w_{\ell}'(x)=-w_{r}'(x)$, from (\ref{diff_error}) we can infer
\begin{equation}
\begin{aligned}
\left \| f'(x)-s'(x) \right \|_{L_{\infty}[-1,1]} \leq \max \lp \left \| f'(x)-s_{\ell}'(x) \right \|_{L_{\infty}[-1,t]} , \left \| f'(x)-s_{r}'(x) \right \|_{L_{\infty}[-t,1]} \rp& \\
+\left \| w_{\ell}'(x) \right \|_{L_{\infty}[-t,t]} \max \lp \left \| f(x)-s_{\ell}(x)\right \|_{L_{\infty}[-t,t]}
, \left \| f(x)-s_{r}(x)\right \|_{L_{\infty}[-t,t]} \rp. &
\end{aligned}
\label{diff_error_2}
\end{equation}
We have that $x=0$ is a critical point of $w_{\ell}'(x)$ and for $t<0.4$ it can be shown that the maximum of $\left |w_{\ell}'(x) \right |$ occurs at $x=0$. Since
\begin{equation}
w_{\ell}'(0) = -\frac{(1+t)^2}{t^2 (2+t)^2},
\end{equation}
we can infer that $\left \| w_{\ell}'(x) \right \|_{L_{\infty}[-t,t]} = \Theta(t^{-2})$ as $t\to 0$. The norm of the Chebyshev differentiation operator is $\Theta(n^2)$ (for $n$ nodes), implying that the two terms on the right-hand side of (\ref{diff_error_2}) are balanced if $t^{-2}=\Theta(n^2)$, or equivalently $t = \Theta \lp \frac{1}{n} \rp$. A simple example of a split can be seen in Figure~\ref{ARCTAN}.

\begin{figure}[!htb]
\centering
\includegraphics[scale = 0.5]{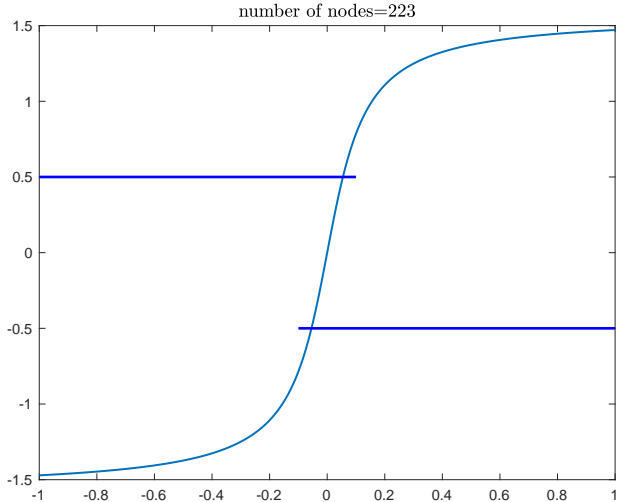}
\caption{Plot of the PU approximation with overlap parameter $t=0.1$ for $f(x)=\arctan \lp x/0.1 \rp$, where the thick lines represent the domains of the left and right approximation. Here $\left \| f(x)-s(x) \right \|_{L_{\infty}[-1,1]}= 2.4 \mathrm{e}{-15}$ and $\left \| f'(x)-s'(x) \right \|_{L_{\infty}[-1,1]}= 1.7 \mathrm{e}{-13}$.}
\label{ARCTAN}
\end{figure}

\section{Recursive algorithm}
\label{PUM_recurse}
In order to allow for adaptation to specific features of $f(x)$, we next describe a recursive bisection algorithm that works similarly to Chebfun's splitting algorithm \cite{driscoll2014optimal} and to that of \cite{tobor2006reconstructing}. Suppose we want to construct a PU approximation $s_{[a,b]}(x)$ on the interval $[a,b]$ using Chebyshev interpolants on the patches. If $f(x)$ can be resolved by a Chebyshev interpolant $s(x)$ of length $\nmax$ on $[a,b]$ then
\begin{align}
s_{[a,b]}(x) = s(x).
\end{align}
Otherwise we split the interval into two overlapping domains and blend the results as in (\ref{PUM2}):
\begin{align}
s_{[a,b]}(x) = w_{\ell}(x)
s_{\left [a,a+\delta \right ]}(x)+ w_{r}(x) s_{\left [b-\delta,b\right ]}(x),
\end{align}
where $w_{\ell},w_{r}$ are the PU weight functions defined in (\ref{PUW}) (but defined for $[a,b]$, and $\delta= (1+t) \lp \frac{a+b}{2} \rp$).

We define a binary tree $T$ with each node $\nu$ having the following properties:
\begin{itemize}
\item interval($\nu$):=the domain of the patch
\item \child{0}($\nu$),\child{1}($\nu$):=respective left and right subtrees of $\nu$ (if split)
\item \weight{0}($\nu$),\weight{1}($\nu$):=respective left and right weights of $\nu$ (if split)
\item interpolant($\nu$):=Chebyshev interpolant on interval($\nu$) if $\nu$ is a leaf
\item values($\nu$):=values of the function we are approximating at the Chebyshev points of $\nu$.
\end{itemize}
We define root($T$) as the root node of $T$. In Algorithm~\ref{alg2} we formally describe how we refine our splitting; the merge method is described in section~\ref{Merging_sec}. 

\begin{algorithm}[!h]
\caption{splitleaves($\nu$,$\nmax$,$t$)}
\label{alg2}
\begin{algorithmic}
\IF{$\nu$ is a leaf and $f(x)$ cannot be resolved by interpolant($\nu$)}
\STATE Define new nodes $\nu_0$, $\nu_1$.
\STATE $[a,b]$:=interval($\nu$)
\STATE $\delta:= \frac{b-a}{2} \lp 1+t \rp$
\STATE interval($\nu_0$):= $[a,a+\delta]$
\STATE interval($\nu_1$):= $[b-\delta,b]$
\FOR{$k=0,1$}
\STATE \child{k}($\nu$) := $\nu_k$
\ENDFOR
\STATE \weight{0}($\nu$),\weight{1}($\nu$):= weights in (\ref{PUW}) defined for $[a,a+\delta]$,$[b-\delta,b]$
\ELSIF{$\nu$ is a leaf and $f(x)$ can be resolved by a Chebyshev interpolant with degree less than $\nmax$}
\STATE interpolant($\nu$):=minimum degree interpolant $f(x)$ can be resolved by \ARP 
as determined by Chebfun
\ELSE
\FOR{$k=0,1$}
\STATE splitleaves(\child{k}($\nu$),$\nmax$,$t$)
\ENDFOR
\STATE merge($\nu$,$\nmax$)
\ENDIF
\end{algorithmic}
\end{algorithm}

We first initialize the tree $T$ with a single node $\nu$ where interval($\nu$)=$[a,b]$. Next we repeatedly call the splitleaves method until each leaf of $T$ has a Chebyshev interpolant that can resolve $f(x)$ with degree less than $\nmax$, as seen in Algorithm~\ref{alg6}. For each leaf $\nu$ of $T$, sample($T$,$f(x)$) sets values($\nu$) using $f(x)$ . For a leaf $\nu$, we determine if a Chebyshev interpolant can resolve $f(x)$ using Chebfun's standardChop method with values($\nu$) (as described in Section~\ref{sec_cheb}). Using $T$ we can evaluate $s_{[a,b]}(x)$ recursively as demonstrated in Algorithm~\ref{alg3}. 

\begin{algorithm}[!h]
\caption{$T$=refine($\nmax$,$t$,$f(x)$)}
\label{alg6}
\begin{algorithmic}
\STATE Define $T$ as a tree with a single node.
\WHILE{$T$ has unresolved leaves}
\STATE sample($T$,$f(x)$)
\STATE splitleaves(root($T$),$\nmax$,$t$)
\ENDWHILE
\end{algorithmic}
\end{algorithm}

\begin{algorithm}[!h]
\caption{v=eval($\nu$,$x$)}
\label{alg3}
\begin{algorithmic}
\IF{$\nu$ is a leaf}
\STATE $p$:=interpolant($\nu$)
\STATE v:= $p(x)$
\ELSE
\STATE $v_0,v_1$:=0

\STATE $w_0$:=\weight{0}($\nu$)
\STATE $w_1$:=\weight{1}($\nu$)
\FOR{$k=0,1$}
\IF{$x \in$ interval(\child{k}($\nu$))}
\STATE $v_k$:=eval(\child{k}($\nu$),$x$)
\ENDIF
\ENDFOR
\STATE v := $w_0(x)v_0 + w_1(x)v_1$
\ENDIF
\end{algorithmic}
\end{algorithm}

As a simple example, we approximate the function $f(x)=\arctan \lp \frac{x-0.25}{0.001} \rp$ with $n_{\max}=128$. In order to resolve to machine precision, a global Chebyshev interpolant on the interval $[-1,1]$ requires 25743 nodes while our method requires 523. Chebfun with non-overlapping splitting requires 381 nodes. Overlapping splittings will typically require more total nodes while offering the benefit of global smoothness. The result can be seen in Figure~\ref{ARCTAN2}.

\begin{figure}[!htb]
\centering
\includegraphics[scale = 0.5]{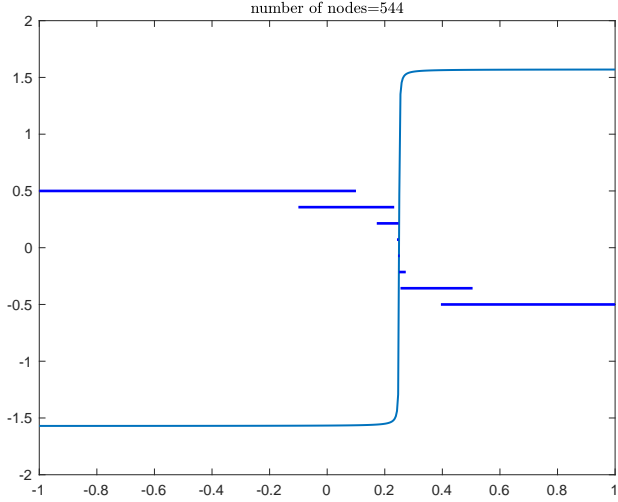}
\caption{Plot of the partition of unity approximation with overlap parameter $t=0.1$ for $f(x)=\arctan \lp (x-0.25)/0.001 \rp$, where the solid blue lines represent the patches.}
\label{ARCTAN2}
\end{figure}

We can deduce from (\ref{POU_UP}) that $s_{[a,b]}(x)$ will approximate $f(x)$. Moreover, our method implicitly creates a PU on the leaves of the tree through the product of the weights at each level.

\begin{theorem}
Let an approximation $s_{[a,b]}(x)$ be as in (\ref{PUM2}). Then the tree that represents $s_{[a,b]}(x)$ implicitly defines a PU $\{w_k(x)\}_{k=1}^M$, where $w_k(x)$ has compact support over the $k$th leaf.
\end{theorem}
\begin{proof}
Suppose that on the domain $[a,b]$ we have PU's $\{w_{\ell k}(x)\}_{k=1}^{M_\ell}$, $\{w_{rk}(x)\}_{k=1}^{M_r}$ for the leaves of the left and right child respectively. We claim that
\begin{equation}
\{w_{\ell}(x) w_{\ell k}(x)\}_{k=1}^{M_\ell} \cup \{w_r(x) w_{r k}(x)\}_{k=1}^{M_r}
\label{UNPU}
\end{equation}
forms a PU over the leaves of the tree. We first observe that $w_{\ell}(x)w_{\ell k}(x)$ will have support in $\supp \lp w_{1k}(x) \rp$, the domain of the respective leaf. This is similarly true for $w_{r}(x)w_{r k}(x)$.

Next suppose that $x \in \sup \lp w_{\ell}(x)\rp \cap \, \sup \lp w_{r}(x)\rp^C$. Then $w_{\ell}(x)=1$ and $w_r(x)=0$, so
\begin{equation}
\sum_{k=1}^{M_\ell} w_{\ell}(x) w_{\ell k}(x)+ \sum_{k=1}^{M_r} w_{r}(x) w_{r k}(x) = \sum_{k=1}^{M_\ell} w_{\ell k}(x) = 1,
\end{equation}
since $\{w_{\ell k}(x)\}_{k=1}^{M_\ell}$ is a PU. This is similarly true if $x \in \sup \lp w_{\ell}(x)\rp^C \cap \sup \lp w_{r}(x)\rp$. Finally if $x \in \sup \lp w_{\ell}(x)\rp \cap \sup \lp w_{r}(x)\rp$ then
\begin{equation}
\begin{aligned}
\sum_{k=1}^{M_\ell} w_{\ell}(x) w_{\ell k}(x)+ \sum_{k=1}^{M_r} w_{r}(x) w_{r k}(x) &=  w_{\ell}(x) \sum_{k=1}^{M_\ell}  w_{\ell k}(x)+ w_{r}(x) \sum_{k=1}^{M_r} w_{r k}(x) \\
&= w_{\ell}(x)+w_{r}(x)=1.
\end{aligned}
\end{equation}
Thus by induction, we have that the product of weights through the binary tree for (\ref{PUM2}) implicitly creates a PU over the leaves.
\end{proof}

\subsection{Merging}
\label{Merging_sec}
As we create the tree we opportunistically merge leaves for greater efficiency. If a particular location in the interval requires a great deal of refinement, the recursive splitting essentially performs a binary search for that location (as was noted about Chebfun splitting in \cite{driscoll2014optimal}). The intermediate splits are not necessarily aiding with resolving the function; they are there just to keep the binary tree full. In Chebfun the recursive splitting phase is followed by a merging phase that discards counterproductive splits. We describe a similar merging operation here, but we allow these merges to take place whenever a leaf splits while its sibling does not, in order to keep the number of leaves from unnecessarily growing exponentially.

\begin{figure}[!htb]
\centering
\adjustbox{valign=t}{ 
\subfloat[Tree before merging. Here ${a_2<b_1}$ and ${a_{22}<b_{21}}$.]{
     \begin{forest}
for tree={circle,draw, l sep=20pt,,scale=0.98}
[ {$\begin{array}{c}
[a,b] \\
s_{[a,b]}(x)
\end{array}$}
    [ {$\begin{array}{c}
 [a,b_1]  \\
s_{[a,b_{1}]}(x) \\
 w_{\ell_1}(x)
\end{array}$},name=left_p]
    [ {{$\begin{array}{c}
    [a_2,b] \\
s_{[a_{2},b]}(x) \\
 w_{r_1}(x)
\end{array}$}} 
      [{$\begin{array}{c}
 [a_2,b_{21}] \\     
s_{[a_2,b_{21}]}(x) \\
 w_{\ell_2}(x)
\end{array}$},name=right_p] 
      [{$\begin{array}{c}
  [a_{22},b] \\    
s_{[a_{22},b]}(x) \\
 w_{r_2}(x)
\end{array}$}] 
  ] 
]
]
\draw[<->,dotted,thick] (right_p) to[out=north west,in=south] (left_p);
\end{forest}
   \label{TRM_1}
 }}\hfill
\adjustbox{valign=t}{ 
\subfloat[Tree after merging.]{
     \begin{forest}
for tree={circle,draw, l sep=20pt,scale=0.98}
[ {$\begin{array}{c}
[a,b]\\
s_{[a,b]}(x)
\end{array}$}
    [ {$\begin{array}{c}
 [a,b_{21}] \\   
s_{[a,b_{21}]}(x) \\
\hat{w}_{\ell_1}(x)
\end{array}$} ]
    [ {$\begin{array}{c}
[a_{22},b] \\    
s_{[a_{22},b] }(x) \\
\hat{w}_{r_1}(x)
\end{array}$} ] 
]
]
\end{forest}
   \label{TRM_2}
 }}
\caption{An example of how leaves are merged, where each node is labeled with its domain, PU approximation and weight.}
\label{Tree_merge}
\end{figure}
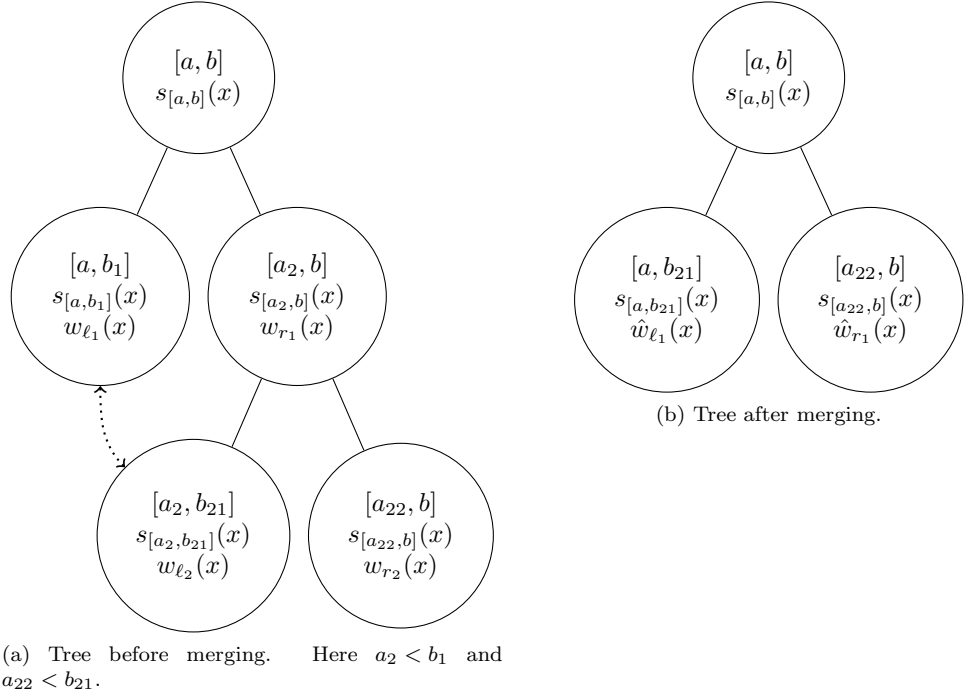

In Figure~\ref{Tree_merge} we illustrate how we merge leaves; the interval $[a,b_{1}]$ is merged with $[a_2,b_{21}]$. Here we decide to merge if $f(x)$ can be resolved with an interpolant with degree less than $\nmax$ on the interval $[a,b_{21}]$. For the new tree we define the left weight $\hat{w}_{\ell_1}(x)$ in Figure~\ref{Tree_merge} as
\begin{align}
\hat{w}_{\ell_1}(x) = \begin{cases} 
                1 & x<a_{22}, \\
                w_{\ell_2}(x) & \text{ otherwise.}        
\end{cases}
\label{pwe}
\end{align}
Since $w_{\ell_2}(x)=1$ for $x<a_{22}$, $\hat{w}_{\ell_1}(x)$ is smooth. For the right weight we use $\hat{w}_{r_1}(x) = w_{r_2}(x)$; these new weights form a PU. The PU approximation
\begin{align}
\hat{s}(x)=w_{\ell 1}(x) s_{[a,b_{1}]}(x)+w_{r 1}(x) s_{[a_2,b_{21}]}(x)
\end{align}
can be used to approximate  $f(x)$ on $[a,b_{21}]$ since $f(x)$ is resolved at the leaves. In this case $s_{[a,b_{21}]}(x)$ is computed from sampling $\hat{s}(x)$. If the degree of $s_{[a,b_{21}]}(x)$ after Chebfun's chopping is less than $\nmax$, we decide to merge. We explain in more detail the merging in Algorithm~\ref{alg5}; here extend($w(x)$,$[a,b]$) piecewise extends the weight $w(x)$ in $[a,b]$ as in (\ref{pwe}).  We show the results for merging in Figure~\ref{MERGE_EXAMPLE} with $f(x)=\frac{1}{x-1.001}$.

\begin{algorithm}[!h]
\caption{merge($\nu$,$\nmax$)}
\label{alg5}
\begin{algorithmic}
\IF{\child{0}($\nu$) and \child{0}(\child{1}($\nu$)) (child and grandchild of $\nu$) are leaves and both of the intervals of the leaves can be resolved on}
\STATE Define a new leaf $\nu_0$
\STATE $p_0(x)$:=interpolant(\child{0}($\nu$))
\STATE $p_1(x)$:=interpolant(\child{1}(\child{0}($\nu$)))
\STATE $w_0(x)$:= \weight{0}($\nu$)
\STATE $w_1(x)$:= \weight{1}($\nu$)
\STATE $\hat{s}(x)$:=$w_0(x)p_0(x)+w_1(x)p_1(x)$
\IF{$\hat{s}(x)$ can be resolved by a Chebyshev interpolant $p(x)$ with degree less than $\nmax$}
\STATE interval($\nu_0$):=interval(\child{0}($\nu$))$\cup$interval(\child{0}(\child{1}($\nu$)))
\STATE interpolant($\nu_0$):=$p(x)$
\STATE points($\nu_0$):=Chebyshev grid of length deg($p(x)$) on $[a_0,b_1]$
\STATE $\hat{w}_0(x)$:=  \weight{0}(\child{1}($\nu$))
\STATE $\hat{w}_1(x)$:=  \weight{1}(\child{1}($\nu$))
\STATE \weight{0}($\nu$):= extend($\hat{w}_0(x)$,interval($\nu_0$))
\STATE \weight{1}($\nu$):= $\hat{w}_1(x)$
\STATE \child{0}($\nu$):= $\nu_0$
\STATE \child{1}($\nu$):= \child{1}(\child{1}($\nu$))
\ENDIF
\ELSIF{\child{1}($\nu$) is a leaf and \child{1}(\child{0}($\nu$)) is a leaf (and exists)}
\STATE inv(merge($\nu$)) (i.e. apply the algorithm, except swap 0 and 1)
\ENDIF
\end{algorithmic}
\end{algorithm}

\begin{figure}[!htb]
\centering
\subfloat[Tree before merging.]{
\includegraphics[scale = 0.35]{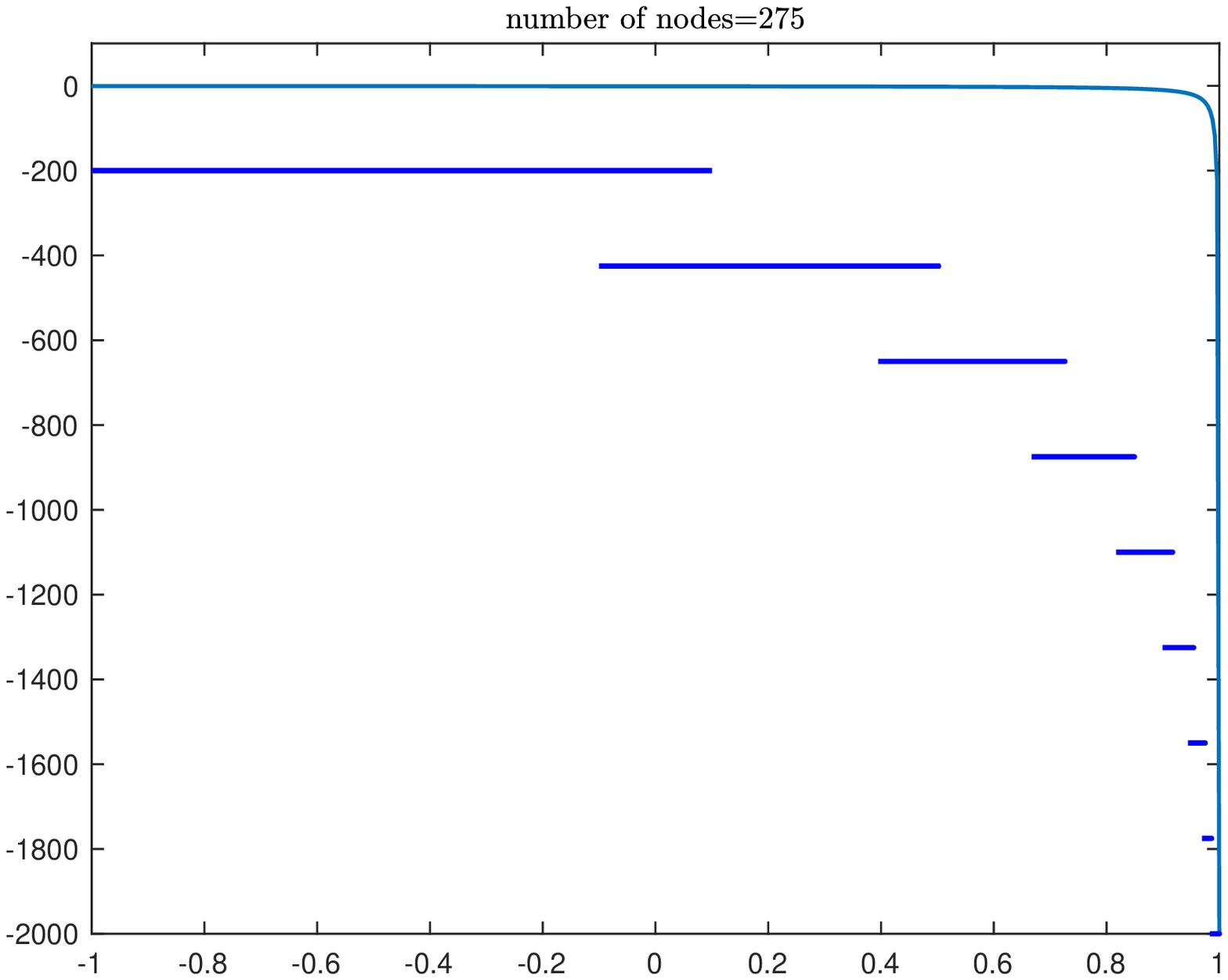}
   \label{MERGEB}
 }
\subfloat[Tree after merging.]{
\includegraphics[scale = 0.35]{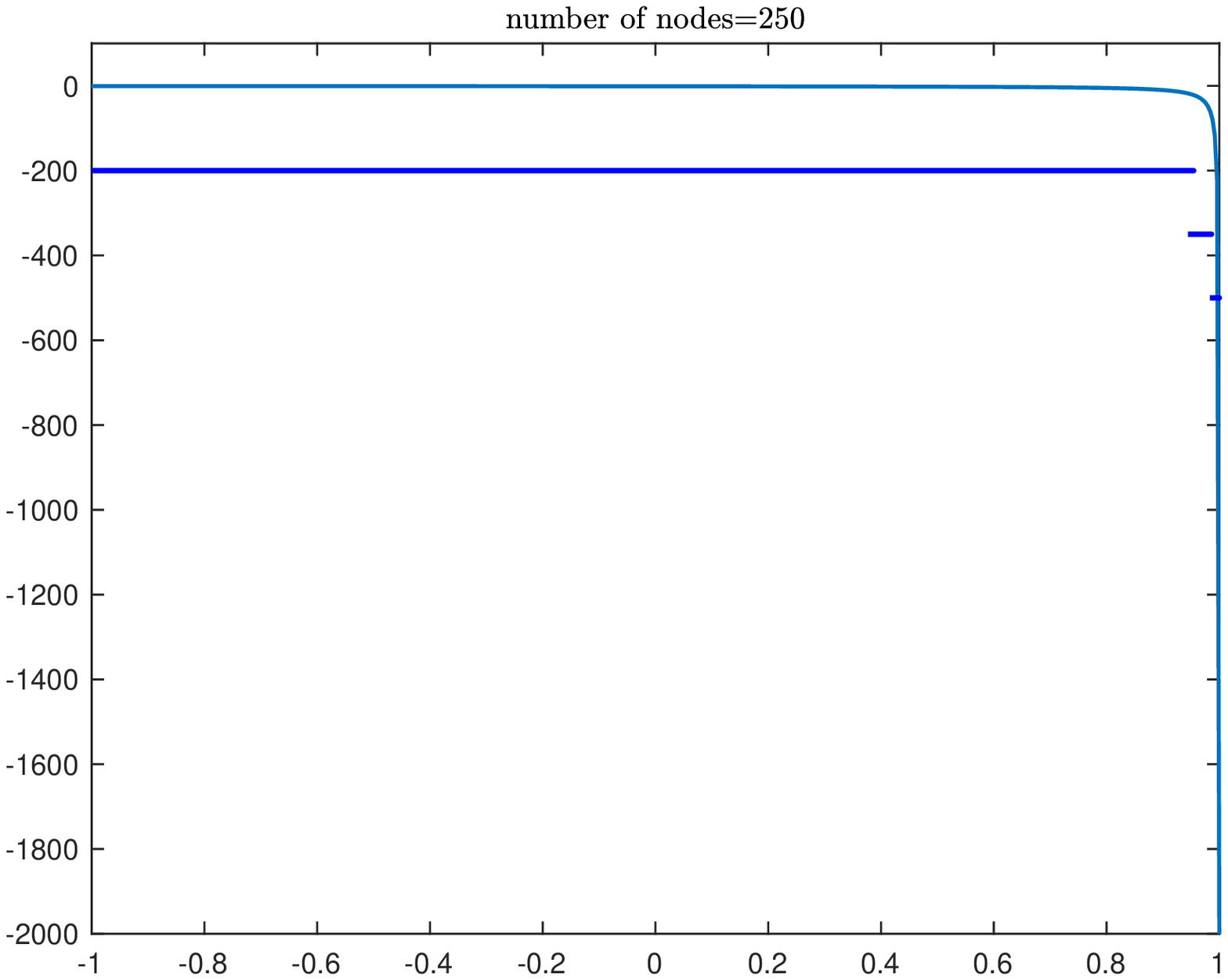}
   \label{MERGEA}
 }
\caption{An example of how the PUM with $t=0.08$, $\nmax=128$ resolves $f(x)=\frac{1}{x-1.0005}$ without merging (a) and after merging (b).
}
\label{MERGE_EXAMPLE}
\end{figure}

\subsection{Differentiation matrices}
\label{PUM_matrix}

Next we demonstrate how to construct a first derivative matrix; higher derivative matrices can be similarly constructed.  Suppose we have constructed a splitting represented with the tree $T$. For each node $\nu$ of the tree, we add the following methods:
\begin{itemize}
\item points($\nu$):= provides the Chebyshev points of the leaves of $\nu$
\item leafpoints($\nu$):= provides the Chebyshev points of $T$ in interval($\nu$) i.e. \newline $\text{points(root($T$))} \cap \text{interval($\nu$)}$
\item pointindex($\nu$):=gives the index of points($\nu$) with respect to the points of the parent of $\nu$ (if $\nu$ is a child)
\item leafpointindex($\nu$):=gives the index of leafpoints($\nu$) with respect to the leafpoints of the parent of $\nu$ (if $\nu$ is a child).
\end{itemize}



Let $[\alpha,\beta]=\text{interval($\nu$)}$. We want to construct matrices $M,D$ such that 
\begin{equation}
\begin{aligned}
M \left . f(x) \right |_{\text{points($\nu$)}} &= \left . s_{[\alpha,\beta]}(x) \right |_{\text{leafpoints($\nu$)}}, \\
D \left . f(x) \right |_{\text{points($\nu$)}} &= \left . \frac{d}{dx} s_{[\alpha,\beta]}(x) \right |_{\text{leafpoints($\nu$)}}.
\end{aligned}
\end{equation}

Let $I_k = \text{interval(\child{k}($\nu$))}$, $w_k(x)=\text{\weight{k}($\nu$)}$ for $k=0,1$. Then
\begin{equation}
\begin{aligned}
\left . s_{[\alpha,\beta]}(x) \right |_{\text{leafpoints($\nu$)}} &= \sum_{k=0}^1 \left . w_k(x) s_{I_k}(x) \right |_{\text{leafpoints($\nu$)}}, \\
\left . \frac{d}{dx} s_{[\alpha,\beta]}(x) \right |_{\text{leafpoints($\nu$)}} &= \sum_{k=0}^1 \lp \left . w_k(x)  \frac{d}{dx} s_{I_k}(x) + \frac{d}{dx} w_k(x)  s_{I_k}(x) \rp \right |_{\text{leafpoints($\nu$)}} .
\end{aligned}
\label{sum_eval}
\end{equation}
Thus we can recursively build up the differentiation matrix through the tree $T$. Due to the support of the weights, for each term in (\ref{sum_eval}) we only need evaluate the approximation $s_{I_k}(x)$ (or its derivative) for $\text{leafpoints($\nu$)} \cap I_k$, i.e. leafpoints(\child{k}($\nu$)). We describe how to construct the differentiation recursively in Algorithm~\ref{alg4}, using MATLAB notation for matrices. At each leaf the interpolation matrix $M$ has entries given by the barycentric interpolation formula based on second-kind Chebyshev points, as produced by the Chebfun command {\tt barymat} \cite{driscoll2015rectangular}.

\begin{algorithm}
\caption{[$M,D$]=diffmatrix($\nu$)}
\label{alg4}
\begin{algorithmic}
\IF{$\nu$ is a leaf}
\STATE $M$:= the Chebyshev barycentric matrix from points($\nu$) to leafpoints($\nu$)
\STATE $D_x$:= Chebyshev differentiation matrix with grid points($\nu$).
\STATE $D$:=$M D_x$.
\ELSE
\STATE $M,D$:=zeros(length(leafpoints($\nu$)),length(points($\nu$)))
\FOR{$k=0,1$}
\STATE [$M_k$,$D_k$]:= diffmatrix(\child{k}($\nu$))
\STATE $M$(leafpointindex(\child{k}($\nu$)),pointindex(\child{k}($\nu$))) = \WRP \text{diag} $\lp \left . w_k \right |_{\text{leafpoints(\child{k}($\nu$))}} \rp$*$M_k$;
\STATE $D$(leafpointindex(\child{k}($\nu$)),pointindex(\child{k}($\nu$))) = \WRP \text{diag} $\lp \left . w_k \right |_{\text{leafpoints(\child{k}($\nu$))}}  \rp$*$D_k$+ \WRP \text{diag} $\lp \frac{d}{dx} \left . w_k \right |_{\text{leafpoints(\child{k}($\nu$))}} \rp$*$M_k$;
\ENDFOR
\ENDIF
\end{algorithmic}
\end{algorithm}

For $x \in [\alpha,\beta]$ we only need to evaluate the local approximations for the patches $x$ belongs to; this implies that the differentiation matrices will be inherently sparse. For example, Figure~\ref{SPARSE_DX} shows the sparsity of the first derivative matrix for the tree generated in Figure~\ref{ARCTAN2}. In this case, we have a sparsity ratio of around 76\%.

\begin{figure}[!htb]
\centering
\includegraphics[scale = 0.5]{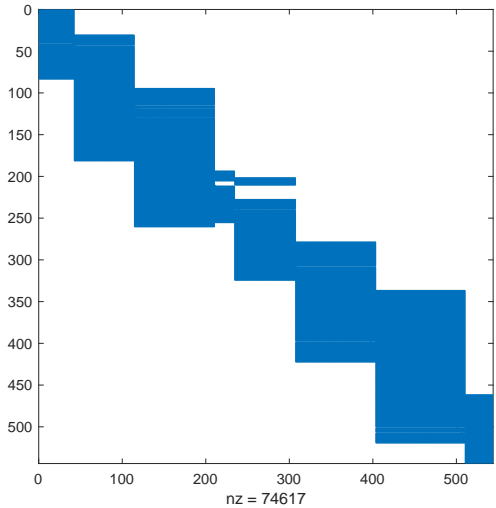}
\caption{Sparsity of the first derivative matrix for the tree generated for Figure~\ref{ARCTAN2}.}
\label{SPARSE_DX}
\end{figure}

\section{PUM for boundary-value problems}
\label{PUM_BVP_SEC}
Our method can be applied to solve linear and nonlinear boundary-value problems. For instance, consider a simple Poisson problem with zero boundary conditions:
\begin{equation}
\begin{aligned}
&u''(x) = f(x) \text{ for }-1<x<1 \\
&u(-1)=0, u(1)=0.
\end{aligned}
\label{simp_pois}
\end{equation}
Suppose that we have differentiation and interpolation matrices $D_{xx}$ and $M$ from section~\ref{PUM_matrix}, $X$ is the set of Chebyshev points over all the leaves, and that $X_I$, $X_B$ are the respective interior and boundary points of $X$. Let $E_{I}$ and $E_{B}$ be the matrices that map a vector to its subvector for the interior and boundary indices respectively. Let $F$ be the vector of values used for the local interpolants (i.e. if we had only two leaves whose interpolants used values $F_1,F_2$, we set $F = [F_1^T  F_2^T]^T$). In order to find a PUM approximation $s(x)$ that approximates (\ref{simp_pois}) we find $F$ by solving the following linear system:
\begin{equation}
\begin{bmatrix}
E_{I} D_{xx} \\[1mm] 
 E_{B} M
\end{bmatrix}
\begin{bmatrix}
E_{I} F \\[1mm]
E_{B} F
\end{bmatrix}
=
\begin{bmatrix}
\left . f \right |_{X_I} \\[1mm]
0 \\ 
\end{bmatrix}.
\label{PUM_lin_system}
\end{equation}

Algorithm~\ref{BVP_solve} builds an adaptive solution for the BVP. We first construct a PU approximation $s(x)$ by solving the discretized system in (\ref{PUM_lin_system}). Sampling with $s(x)$, we use Algorithm~\ref{alg6} to determine if the solution is refined and split leaves that are determined to be unrefined. Here we also allow merging for a node with resolved left and right leaves (i.e., the left and right leaves can be merged back together).

\begin{algorithm}
\caption{$T$=refineBVP($\nmax$,$t$,BVP)}
\label{alg7}
\begin{algorithmic}
\STATE Define $T$ as a tree with a single node with the domain of the BVP.
\WHILE{$T$ has unrefined leaves}
\STATE Find values for the interpolants $F$ of the leaves of $T$ by solving a discretized \ARP system defined by the interpolation and differentiation matrices of $T$.
\STATE sample($T$,$F$) 
\STATE $s(x) = \text{eval}(\text{root}(T),x)$ (the PU approximation)
\STATE sample($T$,$s(x)$)
\STATE splitleaves(root($T$),$\nmax$,$t$)
\ENDWHILE
\end{algorithmic}
\label{BVP_solve}
\end{algorithm}

\subsection{BVP examples}
\label{BVP_SEC}
%
%
%

We solve the stationary Burgers equation on the interval $[0,1]$ with Robin boundary conditions \cite{reyna1995exponentially}:
\begin{equation}
\begin{aligned}
&\nu u''(x)-u(x)u'(x)=0 \\
&\nu u'(0)-\kappa(u(0)-\alpha)=0 \\
&\nu u'(1)+\kappa(u(1)+\alpha)=0
\end{aligned}
\label{PUM_nlin_system}
\end{equation}
which has nontrivial solution
\begin{equation}
u(x)=-\beta \tanh \lp \frac{1}{2} \beta \nu^{-1} \lp x-\frac{1}{2} \rp \rp
\label{true_sol}
\end{equation}
where $\beta$ satisfies
\begin{equation}
-\frac{1}{2} \beta^2 \text{sech}^2 \lp \frac{1}{4} \beta \nu^{-1} \rp+\kappa \left [ \alpha-\beta \text{tanh} \lp \frac{1}{4} \beta \nu^{-1} \rp \right ]=0.
\end{equation}

We choose $\nu=5 \times 10^{-3},\alpha=1$, and $\kappa=2$. We use {\tt fsolve} in MATLAB to solve the BVP, supplying the Jacobian of the discretized nonlinear system. Starting with a linear guess $u(x)=0$, we update the solution from the latest solve (i.e. if the solution $s(x)$ from Algorithm~\ref{BVP_solve} is determined to be unresolved, we use it as the next initial guess). For this problem we set the Chebfun chopping tolerance to $10^{-10}$. Our solution was resolved to the tolerance we set after four nonlinear solves; as seen in Figure~\ref{NLIN_EXAMPLE}, the final approximation had 298 nodes and the absolute error was less than $10^{-4}$ as seen in Figure~\ref{NLIN_EXAMPLE}. On a machine with processor 2.6 GHz Intel Core i5, the solution was found in 1.3 seconds.

\begin{figure}[!htb]
\centering
\subfloat[Solution with subintervals plotted.]{
\includegraphics[scale = 0.35]{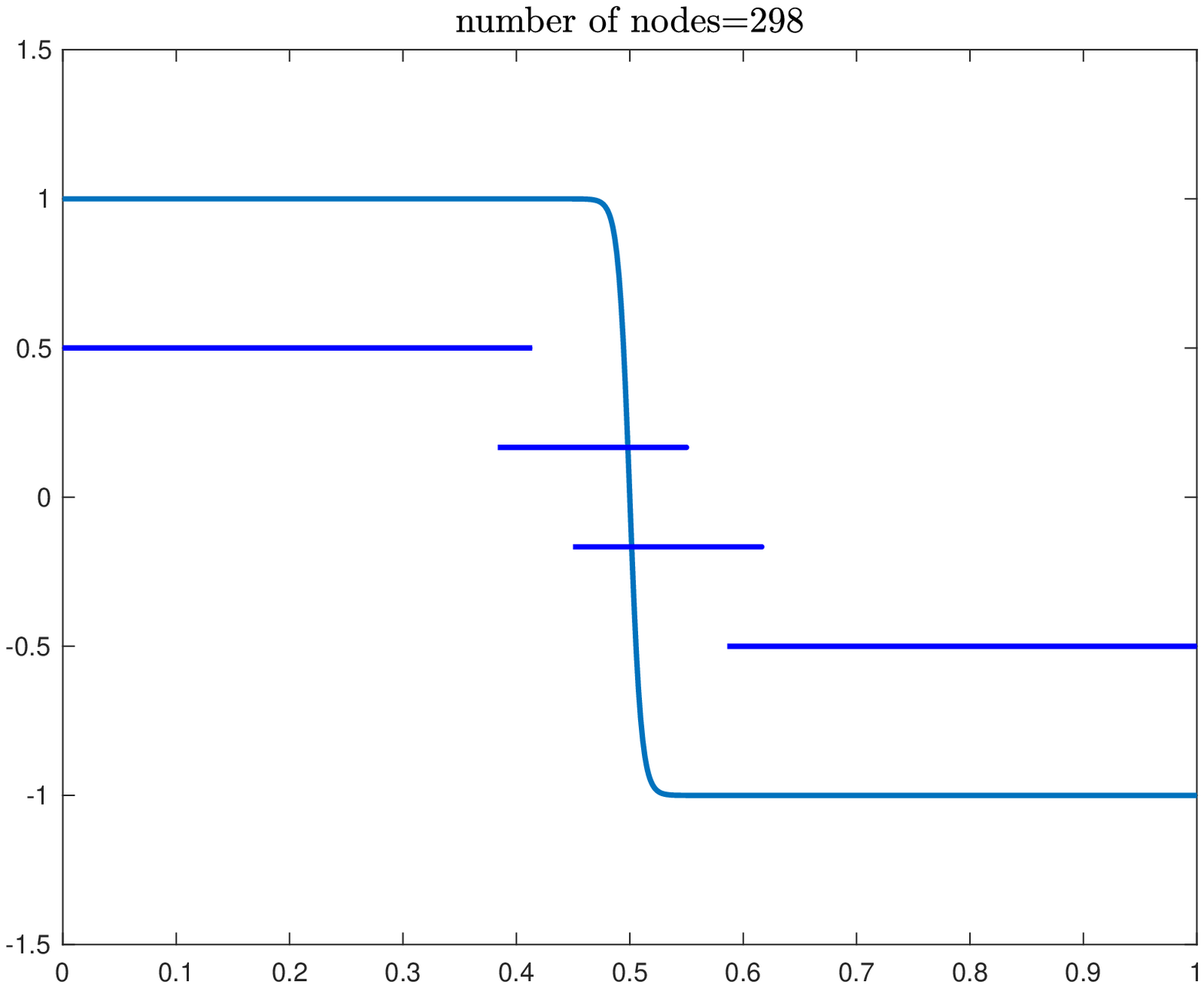}
   \label{NLIN_EXAMPLEA}
 }
\subfloat[Plot of the error.]{
\includegraphics[scale = 0.35]{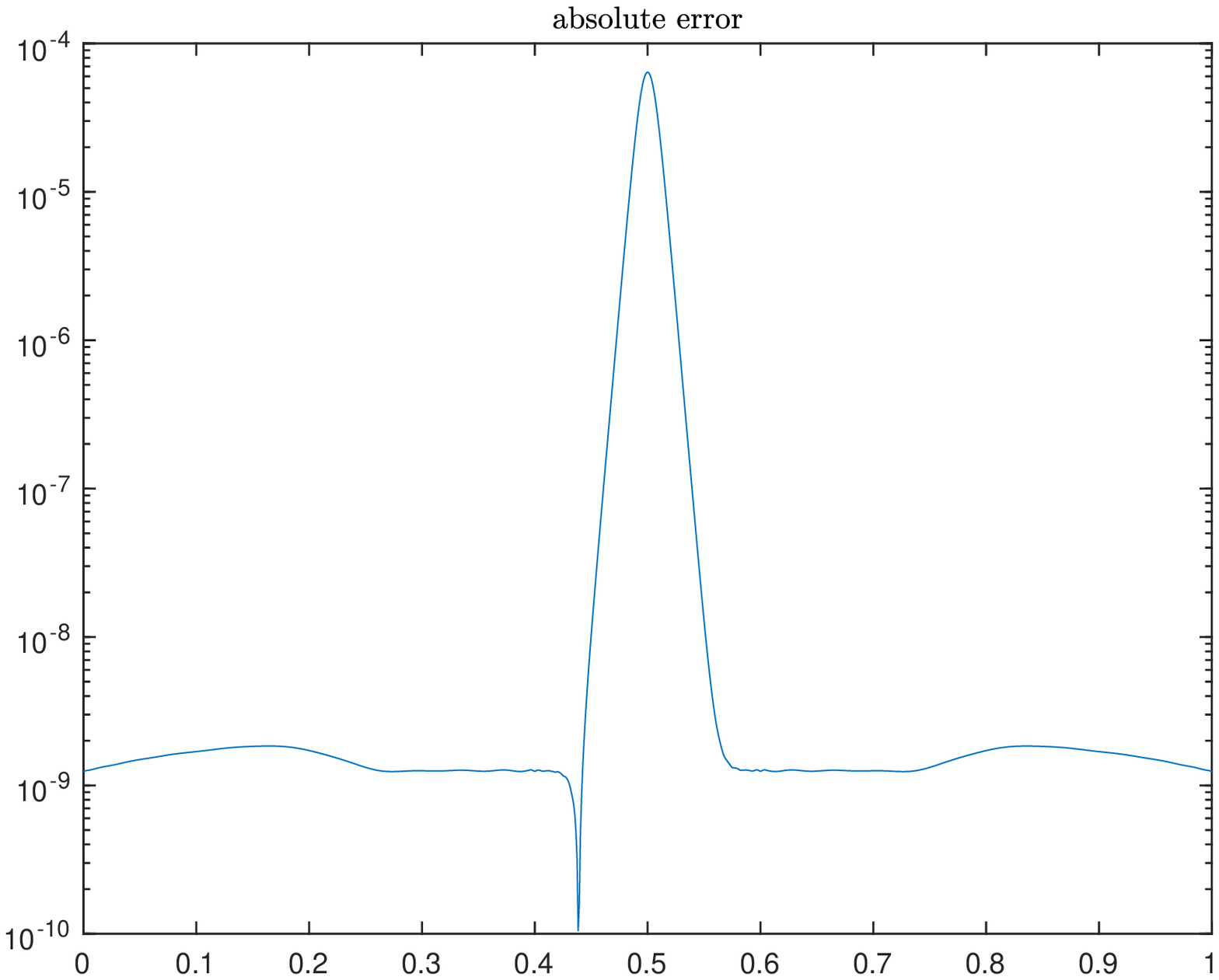}
   \label{NLIN_EXAMPLEB}
 }
\caption{Numerical solution using the PU method and residual for the BVP (\ref{PUM_nlin_system}) with $\nu=5 \times 10^{-3}$, $t=0.1$, and $\nmax = 128$.}
\label{NLIN_EXAMPLE}
\end{figure}

We preformed a similar experiment but instead used global Chebyshev interpolants. We adapt by increasing the degree of the polynomial from $n$ to $\text{floor}(1.5 n)$, starting with $n=128$. We stop when we have a solution that is refined to the tolerance $10^{-10}$ (same as before). Both the solution and residual are in Figure~\ref{GNLIN_EXAMPLE}; here we have the absolute error is higher at 1.8e-2. The solution took 3.2 minutes on the same machine. There are two main reasons why the global solution performs much slower. First, in order to resolve the true solution with the tolerance $10^{-10}$, the global Chebyshev solution requires 766 nodes versus 300 for the PU approximation. Secondly, when adapting with the PUM, if a leaf is determined to be refined, the number of nodes is reduced as dictated in Algorithm~\ref{alg2} and the leaf is not split in further iterations. This keeps the total number of nodes lower while adapting.

\begin{figure}[!htb]
\centering
\subfloat[Solution with subintervals plotted.]{
\includegraphics[scale = 0.35]{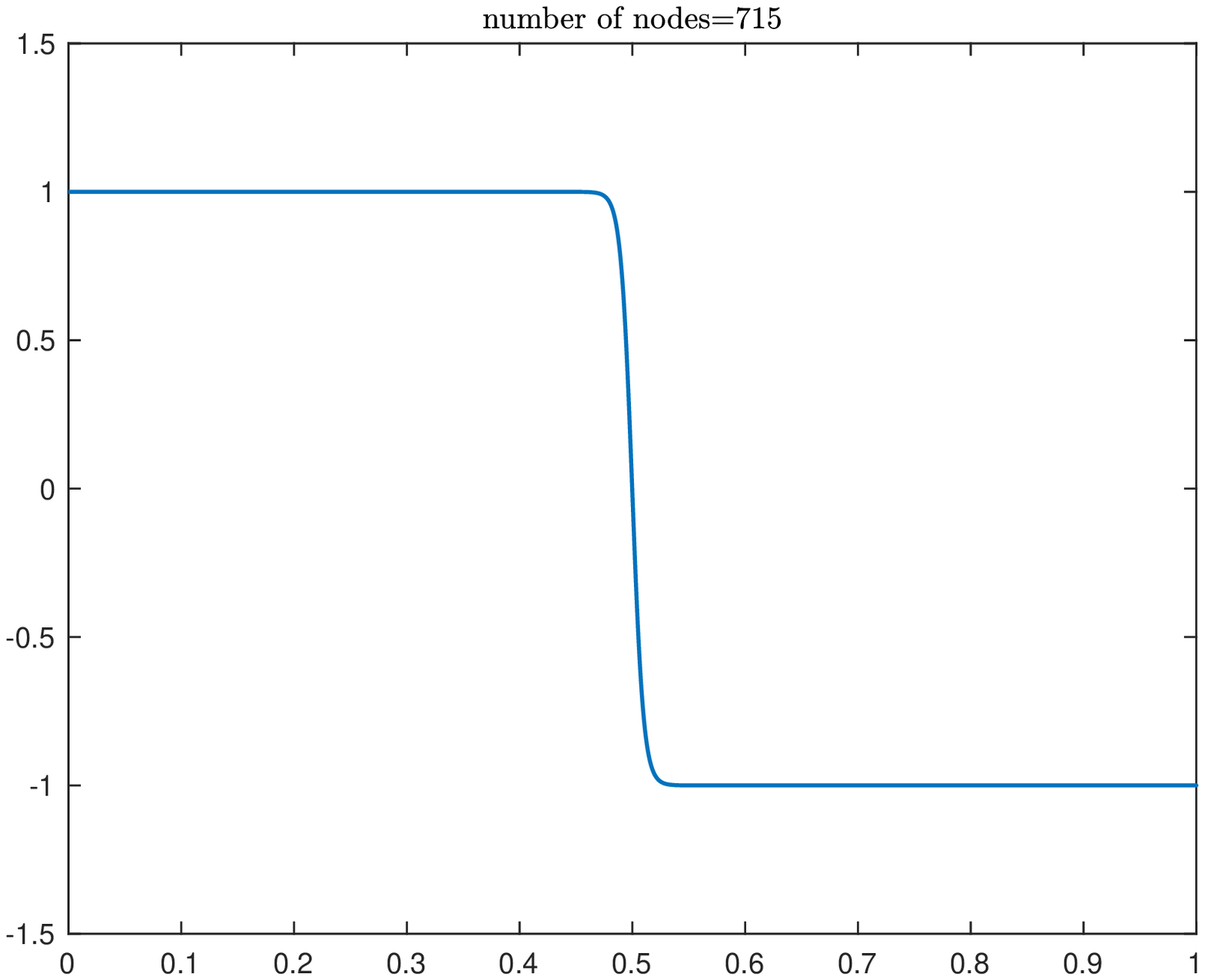}
   \label{GNLIN_EXAMPLEA}
 }
\subfloat[Plot of the error.]{
\includegraphics[scale = 0.35]{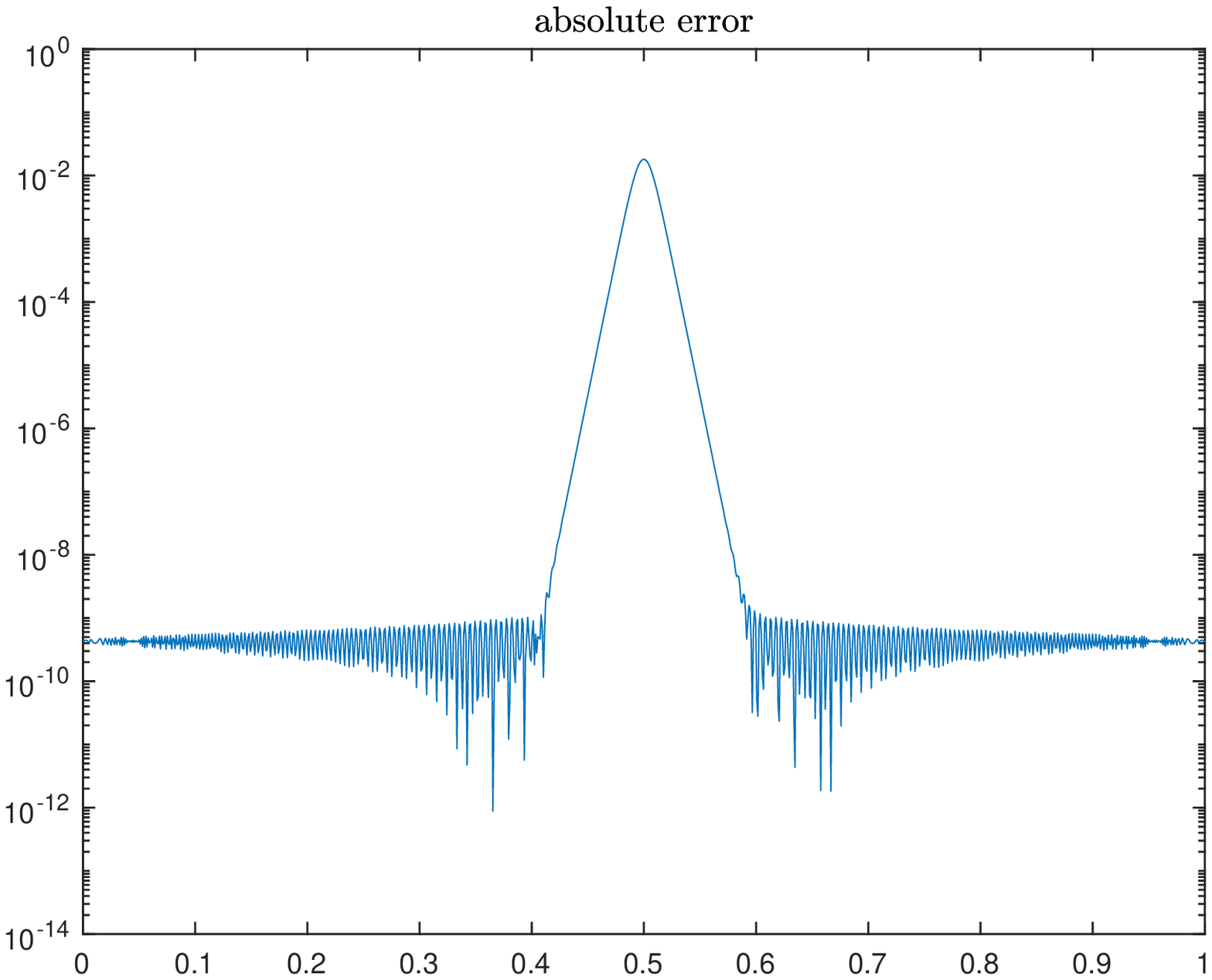}
   \label{GNLIN_EXAMPLEB}
 }
\caption{Numerical solution using the global Chebyshev method and residual for the BVP (\ref{PUM_nlin_system}) with $\nu=5 \times 10^{-3}$.
}
\label{GNLIN_EXAMPLE}
\end{figure}

\section{Discussion}
Our method offers a simple way to adaptively construct infinitely smooth approximations of functions that are given explicitly or that solve BVPs. By recursively constructing the PU weights with the binary tree, we avoid the need to determine the neighbors of each patch (as would be needed with the standard Shepard's PU weights). While this is not a serious issue in one dimension, the complexity of how the patches overlap increases with higher dimension. For example, in 2D we could build a similar method on a box where we use tensor product Chebyshev approximations. We would refine by splitting the box into two overlapping parts (either in $x$ or $y$) and recursively build a binary tree. We similarly would define partition of unities for each of the splits. If we used infinitely smooth weights at the splits, the 2D PU approximation will be infinitely smooth as well.

Our method leaves room for improvement. For instance, while merging helps reduce the number of nodes, in cases where we have a singularity right above the split the PU method over-resolves in the overlap; this can be seen in Figure~\ref{ARCTAN2}. The source of the problem is that patches may be adjacent in space but not in the tree. This could be resolved by a more robust merging algorithm. Alternatively we could determine an optimal splitting location through a Chebyshev-Pad\'{e} approximation as in \cite{driscoll2014optimal}, but the PU adds a layer of complexity since we must optimize not just for the splitting location but the size of the overlap.

Additionally it is possible to construct weights that are not $C^{\infty}$ but have smaller norms in their derivatives. For instance,
\begin{equation}
\begin{aligned}
w_{\ell}(x) &= \begin{cases}
1 & x \leq -t \\
\frac{1}{4t^3} x^3 - \frac{3}{4 t} x+\frac{1}{2} & -t\leq x \leq t \\
0 & x>t
 \end{cases} \\
 w_{r}(x) &= 1-w_{\ell}(x)
\end{aligned}
\end{equation}
defines a $C^1[-1,1]$ piecewise cubic partition of unity, where $\| w_{\ell}'(x)\|_{\infty} = \frac{3}{4t}$. If a BVP requires higher smoothness, we could similarly construct a higher degree polynomial for the weights.

\bibliographystyle{siamplain}
\bibliography{PUPaper.bib}
\end{document}

%% file: ex_shared.tex
\usepackage{lipsum}
\usepackage{amsfonts}
\usepackage{graphicx}
\usepackage{epstopdf}
\usepackage{algorithmic}
\ifpdf
  \DeclareGraphicsExtensions{.eps,.pdf,.png,.jpg}
\else
  \DeclareGraphicsExtensions{.eps}
\fi

\newcommand{\TheTitle}{An adaptive partition of unity method for Chebyshev polynomial interpolation} 
\newcommand{\TheAuthors}{Kevin W. Aiton, Tobin A. Driscoll}

\headers{Adaptive partition of unity}{\TheAuthors}

\title{{\TheTitle}\thanks{Submitted to the editors \today.
\funding{This research was supported by National Science Foundation grant DMS-1412085.}}}

\author{
  Kevin W. Aiton, Tobin A. Driscoll
}

\usepackage{amsopn}


%% file: PUPaper.bbl
\begin{thebibliography}{10}

\bibitem{Aurentz:2017:CCS:3034774.2998442}
{\sc J.~L. Aurentz and L.~N. Trefethen}, {\em Chopping a {C}hebyshev series},
  ACM Trans. Math. Softw., 43 (2017), pp.~33:1--33:21,
  \href{http://dx.doi.org/10.1145/2998442} {doi:10.1145/2998442},
  \url{http://doi.acm.org/10.1145/2998442}.

\bibitem{battles2004extension}
{\sc Z.~Battles and L.~N. Trefethen}, {\em An extension of {MATLAB} to
  continuous functions and operators}, SIAM J. Sci. Comp., 25 (2004),
  pp.~1743--1770.

\bibitem{doi:10.1137/S106482759732678X}
{\sc X.-C. Cai and M.~Sarkis}, {\em A restricted additive schwarz
  preconditioner for general sparse linear systems}, SIAM Journal on Scientific
  Computing, 21 (1999), pp.~792--797,
  \href{http://dx.doi.org/10.1137/S106482759732678X}
  {doi:10.1137/S106482759732678X},
  \url{https://doi.org/10.1137/S106482759732678X},
  \href{http://arxiv.org/abs/https://doi.org/10.1137/S106482759732678X}
  {arXiv:https://doi.org/10.1137/S106482759732678X}.

\bibitem{driscoll2008chebop}
{\sc T.~A. Driscoll, F.~Bornemann, and L.~N. Trefethen}, {\em The chebop system
  for automatic solution of differential equations}, BIT Numerical Mathematics,
  48 (2008), pp.~701--723.

\bibitem{driscoll2015rectangular}
{\sc T.~A. Driscoll and N.~Hale}, {\em Rectangular spectral collocation}, IMA
  Journal of Numerical Analysis, 36 (2015), pp.~108--132.

\bibitem{driscoll2014optimal}
{\sc T.~A. Driscoll and J.~Weideman}, {\em Optimal domain splitting for
  interpolation by {C}hebyshev polynomials}, SIAM Journal on Numerical
  Analysis, 52 (2014), pp.~1913--1927.

\bibitem{franke1980smooth}
{\sc R.~Franke and G.~Nielson}, {\em Smooth interpolation of large sets of
  scattered data}, Internat. J. Numer. Methods Engrg., 15 (1980),
  pp.~1691--1704.

\bibitem{griebel2000particle}
{\sc M.~Griebel and M.~A. Schweitzer}, {\em A particle-partition of unity
  method for the solution of elliptic, parabolic, and hyperbolic {PDEs}}, SIAM
  J. Sci. Comp., 22 (2000), pp.~853--890.

\bibitem{mclain1976two}
{\sc D.~H. McLain}, {\em Two dimensional interpolation from random data}, The
  Computer Journal, 19 (1976), pp.~178--181.

\bibitem{pachon2010piecewise}
{\sc R.~Pach{\'o}n, R.~B. Platte, and L.~N. Trefethen}, {\em Piecewise-smooth
  chebfuns}, IMA journal of {N}umerical {A}nalysis, 30 (2010), pp.~898--916.

\bibitem{reyna1995exponentially}
{\sc L.~G. Reyna and M.~J. Ward}, {\em On the exponentially slow motion of a
  viscous shock}, Communications on Pure and Applied Mathematics, 48 (1995),
  pp.~79--120.

\bibitem{safdari2015radial}
{\sc A.~Safdari-Vaighani, A.~Heryudono, and E.~Larsson}, {\em A radial basis
  function partition of unity collocation method for convection--diffusion
  equations arising in financial applications}, J. Sci. Comp., 64 (2015),
  pp.~341--367.

\bibitem{shepard1968two}
{\sc D.~Shepard}, {\em A two-dimensional interpolation function for
  irregularly-spaced data}, in Proceedings of the 1968 23rd ACM national
  conference, ACM, 1968, pp.~517--524.

\bibitem{tobor2006reconstructing}
{\sc I.~Tobor, P.~Reuter, and C.~Schlick}, {\em Reconstructing multi-scale
  variational partition of unity implicit surfaces with attributes}, Graphical
  Models, 68 (2006), pp.~25--41.

\bibitem{trefethen2000spectral}
{\sc L.~N. Trefethen}, {\em Spectral {M}ethods in MATLAB}, vol.~10, {SIAM},
  2000.

\bibitem{Trefethen2013}
{\sc L.~N. Trefethen}, {\em Approximation Theory and Approximation Practice},
  Society for Industrial and Applied Mathematics, 2013.

\bibitem{wendland2004scattered}
{\sc H.~Wendland}, {\em Scattered Data Approximation}, Cambridge University
  Press, 2004.

\end{thebibliography}
